\documentclass{amsart}
\usepackage{latexsym}
\usepackage{amssymb, amsbsy, amsthm, amsmath, amstext, amsopn, verbatim}
\usepackage[color,all]{xy}
\usepackage{amsfonts}
\usepackage{amscd}
\usepackage{mathrsfs}
\usepackage{verbatim}
\usepackage{xcolor}
\usepackage{tabularx}
\usepackage{mathalfa}

\input xy
\xyoption{all}

\usepackage{parcolumns}

\newtheorem{thm}{Theorem} [section]
\newtheorem{lemma}[thm]{Lemma}

\newtheorem{corollary}[thm]{Corollary}
\newtheorem{prop}[thm]{Proposition}

\newtheorem{assumption}[thm]{Assumption}

\newtheorem*{rough-thm-1}{Rough Version of Vanishing Theorem}
\newtheorem*{rough-thm-2}{Rough Version of Exactness Theorem}

\newtheorem*{main example}{Main Example}

\theoremstyle{definition}

\newtheorem*{basic convention}{Basic Conventions}

\newtheorem{convention}[thm]{Convention}

\theoremstyle{remark}

\newtheorem{remark}[thm]{Remark}

\begin{document}

\numberwithin{equation}{section}

\newcommand{\hs}{\mbox{\hspace{.4em}}}
\newcommand{\ds}{\displaystyle}
\newcommand{\bd}{\begin{displaymath}}
\newcommand{\ed}{\end{displaymath}}
\newcommand{\bcd}{\begin{CD}}
\newcommand{\ecd}{\end{CD}}

\newcommand{\proj}{\operatorname{Proj}}
\newcommand{\bproj}{\underline{\operatorname{Proj}}}
\newcommand{\spec}{\operatorname{Spec}}
\newcommand{\bspec}{\underline{\operatorname{Spec}}}
\newcommand{\pline}{{\mathbf P} ^1}
\newcommand{\pplane}{{\mathbf P}^2}
\newcommand{\coker}{{\operatorname{coker}}}
\newcommand{\ldb}{[[}
\newcommand{\rdb}{]]}

\newcommand{\Sym}{\operatorname{Sym}^{\bullet}}
\newcommand{\Symp}{\operatorname{Sym}}
\newcommand{\Pic}{\operatorname{Pic}}
\newcommand{\AAut}{\operatorname{Aut}}
\newcommand{\PAut}{\operatorname{PAut}}

\newcommand{\too}{\twoheadrightarrow}
\newcommand{\C}{{\mathbf C}}
\newcommand{\cA}{{\mathcal A}}
\newcommand{\cS}{{\mathcal S}}
\newcommand{\cV}{{\mathcal V}}
\newcommand{\cM}{{\mathcal M}}
\newcommand{\bA}{{\mathbf A}}
\newcommand{\aline}{\mathbb{A}^1}
\newcommand{\cB}{{\mathcal B}}
\newcommand{\cC}{{\mathcal C}}
\newcommand{\cD}{{\mathcal D}}
\newcommand{\D}{{\mathcal D}}
\newcommand{\cs}{{\mathbf C} ^*}
\newcommand{\boldc}{{\mathbf C}}
\newcommand{\cE}{{\mathcal E}}
\newcommand{\cF}{{\mathcal F}}
\newcommand{\cG}{{\mathcal G}}
\newcommand{\G}{{\mathbf G}}
\newcommand{\fg}{{\mathfrak g}}
\newcommand{\ft}{\mathfrak t}
\newcommand{\bH}{{\mathbf H}}
\newcommand{\cH}{{\mathcal H}}
\newcommand{\cI}{{\mathcal I}}
\newcommand{\cJ}{{\mathcal J}}
\newcommand{\cK}{{\mathcal K}}
\newcommand{\cL}{{\mathcal L}}
\newcommand{\baL}{{\overline{\mathcal L}}}
\newcommand{\M}{{\mathcal M}}
\newcommand{\bM}{{\mathbf M}}
\newcommand{\bm}{{\mathbf m}}
\newcommand{\cN}{{\mathcal N}}
\newcommand{\theo}{\mathcal{O}}
\newcommand{\cP}{{\mathcal P}}
\newcommand{\cR}{{\mathcal R}}
\newcommand{\boldp}{{\mathbf P}}
\newcommand{\boldq}{{\mathbf Q}}
\newcommand{\bbL}{{\mathbf L}}
\newcommand{\cQ}{{\mathcal Q}}
\newcommand{\cO}{{\mathcal O}}
\newcommand{\Oo}{{\mathcal O}}
\newcommand{\OX}{{\Oo_X}}
\newcommand{\OY}{{\Oo_Y}}
\newcommand{\otY}{{\underset{\OY}{\ot}}}
\newcommand{\otX}{{\underset{\OX}{\ot}}}
\newcommand{\cU}{{\mathcal U}}
\newcommand{\cX}{{\mathcal X}}
\newcommand{\cW}{{\mathcal W}}
\newcommand{\boldz}{{\mathbf Z}}
\newcommand{\cZ}{{\mathcal Z}}
\newcommand{\qgr}{\operatorname{qgr}}
\newcommand{\gr}{\operatorname{gr}}
\newcommand{\coh}{\operatorname{coh}}
\newcommand{\End}{\operatorname{End}}
\newcommand{\Hom}{\operatorname{Hom}}
\newcommand{\uHom}{\underline{\operatorname{Hom}}}
\newcommand{\uHomY}{\uHom_{\OY}}
\newcommand{\uHomX}{\uHom_{\OX}}
\newcommand{\Ext}{\operatorname{Ext}}
\newcommand{\bExt}{\operatorname{\bf{Ext}}}
\newcommand{\Tor}{\operatorname{Tor}}

\newcommand{\inv}{^{-1}}
\newcommand{\airtilde}{\widetilde{\hspace{.5em}}}
\newcommand{\airhat}{\widehat{\hspace{.5em}}}
\newcommand{\nt}{^{\circ}}
\newcommand{\del}{\partial}

\newcommand{\supp}{\operatorname{supp}}
\newcommand{\GK}{\operatorname{GK-dim}}
\newcommand{\hd}{\operatorname{hd}}
\newcommand{\id}{\operatorname{id}}
\newcommand{\res}{\operatorname{res}}
\newcommand{\lrar}{\leadsto}
\newcommand{\im}{\operatorname{Im}}
\newcommand{\HH}{\operatorname{H}}
\newcommand{\TF}{\operatorname{TF}}
\newcommand{\Bun}{\operatorname{Bun}}
\newcommand{\Hilb}{\operatorname{Hilb}}
\newcommand{\Fact}{\operatorname{Fact}}
\newcommand{\F}{\mathcal{F}}
\newcommand{\nthord}{^{(n)}}
\newcommand{\Aut}{\underline{\operatorname{Aut}}}
\newcommand{\Gr}{\operatorname{Gr}}
\newcommand{\Fr}{\operatorname{Fr}}
\newcommand{\GL}{\operatorname{GL}}
\newcommand{\gl}{\mathfrak{gl}}
\newcommand{\SL}{\operatorname{SL}}
\newcommand{\ff}{\footnote}
\newcommand{\ot}{\otimes}
\def\Ext{\operatorname {Ext}}
\def\Hom{\operatorname {Hom}}
\def\Ind{\operatorname {Ind}}
\def\bbZ{{\mathbb Z}}

\newcommand{\nc}{\newcommand}
\newcommand{\on}{\operatorname}
\nc{\cont}{\on{cont}}
\nc{\rmod}{\on{mod}}
\nc{\Mtil}{\widetilde{M}}
\nc{\wb}{\overline}
\nc{\wt}{\widetilde}
\nc{\wh}{\widehat}
\nc{\sm}{\setminus}
\nc{\mc}{\mathcal}
\nc{\mbb}{\mathbb}
\nc{\Mbar}{\wb{M}}
\nc{\Nbar}{\wb{N}}
\nc{\Mhat}{\wh{M}}
\nc{\pihat}{\wh{\pi}}
\nc{\JYX}{\cJ_{Y\leftarrow X}}
\nc{\phitil}{\wt{\phi}}
\nc{\Qbar}{\wb{Q}}
\nc{\DYX}{\D_{Y\leftarrow X}}
\nc{\DXY}{\D_{X\to Y}}
\nc{\dR}{\stackrel{\bbL}{\underset{\D_X}{\ot}}}
\nc{\Winfi}{\cW_{1+\infty}}
\nc{\K}{{\mc K}}
\nc{\unit}{{\bf \on{unit}}}
\nc{\boxt}{\boxtimes}
\nc{\xarr}{\stackrel{\rightarrow}{x}}
\nc{\Cnatbar}{\overline{C}^{\natural}}
\nc{\oJac}{\overline{\on{Jac}}}
\nc{\gm}{{\mathbf G}_m}
\nc{\Loc}{\on{Loc}}
\nc{\Bm}{\operatorname{Bimod}}
\nc{\lie}{{\mathfrak g}}
\nc{\lb}{{\mathfrak b}}
\nc{\lien}{{\mathfrak n}}
\nc{\e}{\epsilon}
\nc{\eu}{\mathsf{eu}}

\nc{\Gm}{{\mathbb G}_m}
\nc{\Gabar}{\wb{\G}_a}
\nc{\Gmbar}{\wb{\G}_m}
\nc{\PD}{{\mathbb P}_{\D}}
\nc{\Pbul}{P_{\bullet}}
\nc{\PDl}{{\mathbb P}_{\D(\lambda)}}
\nc{\PLoc}{\mathsf{MLoc}}
\nc{\Tors}{\on{Tors}}
\nc{\PS}{{\mathsf{PS}}}
\nc{\PB}{{\mathsf{MB}}}
\nc{\Pb}{{\underline{\operatorname{MBun}}}}
\nc{\Ht}{\mathsf{H}}
\nc{\bbH}{\mathbb H}
\nc{\gen}{^\circ}
\nc{\Jac}{\operatorname{Jac}}
\nc{\sP}{\mathsf{P}}
\nc{\sT}{\mathsf{T}}
\nc{\bP}{{\mathbb P}}
\nc{\otc}{^{\otimes c}}
\nc{\Det}{\mathsf{det}}
\nc{\PL}{\on{ML}}
\nc{\sL}{\mathsf{L}}

\nc{\ml}{{\mathcal S}}
\nc{\Xc}{X_{\on{con}}}
\nc{\Z}{{\mathbb Z}}
\nc{\resol}{\mathfrak{X}}
\nc{\map}{\mathsf{f}}
\nc{\gK}{\mathbb{K}}
\nc{\bigvar}{\mathsf{W}}
\nc{\Tmax}{\mathsf{T}^{md}}

\nc{\Cpt}{\mathbb{P}}
\nc{\pv}{e}

\nc{\Qgtr}{Q^{\on{gtr}}}
\nc{\algtr}{\alpha^{\on{gtr}}}
\nc{\Ggtr}{\mathbb{G}^{\on{gtr}}}
\nc{\chigtr}{\chi^{\on{gtr}}}

\newcommand{\la}{\langle}
\newcommand{\ra}{\rangle}
\newcommand{\fm}{\mathfrak m}
\newcommand{\Ms}{\mathfrak M}
\newcommand{\Ma}{\mathfrak M_0}
\newcommand{\ML}{\mathfrak L}
\newcommand{\ev}{\textit{ev}}

\nc{\EF}{\Phi}

\nc{\Cs}{{\mathbb C}^*}

\numberwithin{equation}{section}

\title{Counterexamples to Hyperk\"ahler Kirwan Surjectivity}

\author{Kevin McGerty}
\address{Mathematical Institute\\University of Oxford\\Oxford OX1 3LB, UK}
\email{mcgerty@maths.ox.ac.uk}
\author{Thomas Nevins}
\address{Department of Mathematics\\University of Illinois at Urbana-Champaign\\Urbana, IL 61801 USA}
\email{nevins@illinois.edu}

\begin{abstract}
Suppose that $\mathsf{M}$ is a complete hyperk\"ahler manifold with a compact Lie group $K$ acting via hyperk\"ahler isometries and with hyperk\"ahler moment map $(\mu_{\mathbb{C}}, \mu_{\mathbb{R}}): \mathsf{M}\rightarrow \mathfrak{k}^*\otimes\on{Im}(\mathbb{H})$.   It is a long-standing problem to determine when the hyperk\"ahler Kirwan map
\bd
H^*_K(\mathsf{M},\mathbb{Q}) \longrightarrow H^*(\mathsf{M}/\!\!/\!\!/ K, \mathbb{Q})
\ed 
is surjective.  We show that for each $n\geq 2$, the natural $U(n)$-action on $T^*(SL_n\times\mathbb{C}^n)$ admits a hyperk\"ahler quotient for which the hyperk\"ahler Kirwan map fails to be surjective.   As a tool, we establish a ``K\"ahler $=$ GIT quotient'' assertion for products of cotangent bundles of reductive groups, equipped with the Kronheimer metric, and representations.
\end{abstract}

\maketitle

\section{Introduction}
Suppose that $\mathsf{M}$ is a complete hyperk\"ahler manifold with a compact Lie group $K$ acting via hyperk\"ahler isometries and a corresponding hyperk\"ahler moment map 
\bd
(\mu_{\mathbb C}, \mu_{\mathbb R}): \mathsf{M}\longrightarrow (\mathfrak{k}^*)\otimes \on{Im}(\mathbb{H}).
\ed
 Choosing $\xi\in Z(\mathfrak{k}^*)$ sufficiently generic, we assume that $K$ acts quasi-freely (i.e., with finite stabilizers) on $\mu_{\mathbb{C}}\inv(0)\cap\mu_{\mathbb{R}}\inv(\xi)$ making
 \bd
 \mathsf{M}/\!\!/\!\!/ K := \mu_{\mathbb{C}}\inv(0)\cap\mu_{\mathbb{R}}\inv(\xi)/K
 \ed
 a smooth hyperk\"ahler orbifold.  The {\em hyperk\"ahler Kirwan map} is the restriction map,
\begin{equation}\label{eq:HKK map}
\xymatrix{
H^*_K(\mathsf{M},\mathbb{Q}) \ar[r]^{\hspace{-3.5em}\kappa} &  H^*_K(\mu_{\mathbb{C}}\inv(0)\cap\mu_{\mathbb{R}}\inv(\xi), \mathbb{Q}) \ar[r]^{\hspace{2em}\cong} & H^*(M /\!\!/\!\!/ K, \mathbb{Q}).
}
\end{equation}
It is a long-standing question when \eqref{eq:HKK map} is surjective.  
The map \eqref{eq:HKK map} is known (cf. \cite{Konno, JKK, FR}, among others) to be surjective for some classes of examples including quiver varieties \cite{McN-HKK}; it is known not to be surjective for some hyperk\"ahler quotients of infinite-dimensional vector spaces (cf. \cite{Hitchin, Hausel, DWWW, CNS}).

\vspace{.3em}

The present paper exhibits an infinite list of examples for which $\mathsf{M}$ is a (finite-dimensional) affine algebraic variety (in complex structure $I$), $G$ is a unitary group, and \eqref{eq:HKK map} is {\em not} surjective.  Our counterexamples are built via familiar methods in complex algebraic geometry, and indeed are closely related to manifolds for which \eqref{eq:HKK map} is known (via \cite{Vasserot}) to be surjective.  

\vspace{.3em}

In fact, the paper combines two largely independent parts to achieve this goal.

\vspace{.3em}

First, let $G_1,\dots, G_n$ be a finite collection of complex reductive groups and let $\mathbb{G} = \prod_i G_i$.  Then $\mathbb{G}^2$ acts on $\mathbb{G}$ via the product of left and right actions (which we call the ``left-right action''); write $\mathbb{K}\subset\mathbb{G}$ for a choice of maximal compact subgroup.  Let $G\subseteq\mathbb{G}^2$ be a reductive subgroup acting on $\mathbb{G}$ via the left-right actions and with maximal compact subgroup $K = \mathbb{K}\cap G$.  Let $V$ be a finite-dimensional representation of $G$ with $K$-invariant Hermitian metric and let $\mathsf{M} = T^*\mathbb{G}\times T^*V$.    Thanks to work of Kronheimer  \cite{Kronheimer} (see also \cite{DancerSwann} for a clear exposition), it is known that $T^*\mathbb{G}$ admits a complete, $\mathbb{K}$-invariant  hyperk\"ahler metric (which we call the ``Kronheimer metric''), and hence $K$ acts on $\mathsf{M}$ by hyperk\"ahler isometries.  Because $\mathsf{M}$ is a cotangent bundle, we obtain canonical complex and real moment maps 
$(\mu_{\mathbb C}, \mu_{\mathbb R}): \mathsf{M} \longrightarrow \mathfrak{g}^*\times \mathfrak{k}^*.$
Choose a character $\chi: G\rightarrow \Gm=\mathbb{C}^*$ and let $\lambda = d\chi:\mathfrak{g}\rightarrow\mathbb{C}$, yielding $-i\lambda \in \Hom(\mathfrak{k},\mathbb{R}) = \mathfrak{k}^*$.  
\begin{thm}[Theorem \ref{kempf-ness thm}]\label{kempf-ness intro}
The GIT and hyperk\"ahler quotients of $\mathsf{M}$ are isomorphic: that is, 
\bd
\mu_{\mathbb{C}}\inv(0)\cap\mu_{\mathbb{R}}\inv(-i\lambda)/K \cong \mu_{\mathbb{C}}\inv(0)/\!\!/_{\chi} G,
\ed
where the right-hand side is the GIT quotient of $\mu_{\mathbb{C}}\inv(0)$ with respect to the character $\chi$.
\end{thm}
Such ``K\"ahler quotient $=$ GIT quotient'' assertions are ubiquitous in the literature and there is a standard approach to proving them: we note  the relevant references \cite{Kempf-Ness, Kirwan, AL, Sjamaar}, and especially the recent papers \cite{Hoskins, Mayrand, Takayama, NT}. 
In particular, it was known to some experts that one could prove (some version of) Theorem \ref{kempf-ness intro} using the standard approach (see \cite{Takayama} for a closely related situation as well as Theorem 2.1 and Section 3.3 of \cite{NT}).  But the standard approach still requires checking a convergence condition on downward Morse flows that does not seem to have been documented before in the generality we need, and we believe that to non-experts (such as ourselves) it is useful to have a complete proof in the literature.

  The second part of the paper then analyzes the hyperk\"ahler, or equivalently by Theorem \ref{kempf-ness intro}, algebraic-symplectic, Kirwan map for $\mathsf{M}$ in the special case of $\mathbb{G} = GL_n = SL_n\times_{\mu_n} \Gm$ acting on $\mathsf{M} = T^*SL_n\times T^*\mathbb{C}^n$ induced from the adjoint action on $SL_n$ and the obvious action on $\mathbb{C}^n$.  In particular, we prove:
\begin{thm}\label{main thm}
For $n\geq 2$, the hyperk\"ahler Kirwan map
\bd
H^*_{U(n)}\big(\mathsf{M}, \mathbb{Q}\big)\longrightarrow H^*\big(\mathsf{M}/\!\!/\!\!/ U(n), \mathbb{Q}\big)
\ed
fails to be surjective.  More specifically, the restriction of the hyperk\"ahler Kirwan map to the cohomology of pure weight with respect to the canonical mixed Hodge structures
\bd
\bigoplus_k W_k H^k_{U(n)}\big(\mathsf{M}, \mathbb{Q}\big) \longrightarrow \bigoplus_k W_kH^k\big(\mathsf{M}/\!\!/\!\!/ U(n), \mathbb{Q}\big)
\ed
 fails to be surjective.
\end{thm}
If we replace $\mathsf{M} = T^*(SL_n\times \mathbb{A}^n)$ in the construction by $T^*(GL_n\times\mathbb{A}^n)$, the hyperk\"ahler quotient, or equivalently (by Theorem \ref{kempf-ness intro}) GIT quotient, becomes (Lemma \ref{Hilb scheme})
$T^*(GL_n\times\mathbb{A}^n)/\!\!/\!\!/ U(n) \cong (\mathbb{C}^*\times \mathbb{C})^{[n]}$, the Hilbert scheme of $n$ points on $\mathbb{C}^*\times \mathbb{C}$.  Our proof of Theorem \ref{main thm} carries out a low-tech comparison of cohomologies of $\mathsf{M}/\!\!/\!\!/G$ with those of $(\mathbb{C}^*\times \mathbb{C})^{[n]}$ to deduce the failure of surjectivity.  
This is the simplest situation we have so far found in which we can prove that the hyperk\"ahler Kirwan map fails to be surjective, but one can use the same technique to produce other examples: for example, replacing $T^*SL_n$ by an analogue associated to a cyclic quiver with $\ell$ nodes. 

We note that hyperk\"ahler Kirwan surjectivity was already known to fail for certain {\em infinite-dimensional} hyperk\"ahler quotients yielding moduli spaces of Higgs bundles: see \cite{Hitchin} (and the relevant discussion in \cite{Hausel}, \cite{DWWW}) as well as \cite{CNS}.\footnote{In particular, our method (using the action of the center $Z(SL_n) = \mu_n$ on cohomology) will be recognized by the reader familiar with the Higgs bundle context (cf. \cite{DWWW, CNS}).} 
Experts knew that there should also exist counterexamples among finite-dimensional hyperk\"ahler quotients, but such examples seem not to have been published.  Thus, their availability beyond a circle of experts was uncertain. 
In any case, lacking a definitive characterization of the image of \eqref{eq:HKK map} in all cases, it seems important to begin to circumscribe the possibilities explicitly.

H. Nakajima noted that the examples we consider fit naturally in a more general context of Coulomb branches associated to $3D$ $N=4$  gauge theories; he also explained that there should probably be many more examples of failure of hyperk\"ahler Kirwan surjectivity to be found among branches of moduli spaces of vacua of such theories: see \cite{NT} for extensive treatment of Coulomb branches of affine quiver gauge theories from the viewpoint of Cherkis bow varieties.
\subsection*{Acknowledgments}
We are grateful to Richard Wentworth for an important conversation; to John D'Angelo and Andrew Dancer for helpful discussions and comments; to Young-Hoon Kiem and Hiraku Nakajima for very helpful remarks on K\"ahler = GIT identifications; and to Hiraku Nakajima for a very informative discussion of Coulomb branches and bow varieties.  This manuscript has been circulated since May 2018; we are grateful to those who have provided comments, including Young-Hoon Kiem, Hiraku Nakajima, and Ciaran O'Neill.
The first author was supported by EPSRC programme grant EI/I033343/1.  The second author was supported by NSF grants DMS-1502125 and DMS-1802094 and a Simons Foundation fellowship.

\begin{convention}\label{convention}
All algebraic varieties and algebraic groups are defined over $\mathbb{C}$.
\end{convention}

\section{Construction of the Quotient}
\subsection{Complex Group Construction}
We begin with a finite list $G_1,\dots, G_n$ of complex reductive groups, and let $\mathbb{G}= \prod_i G_i$.  Each factor $G_i$ inherits the left and right actions of $G_i$ by $(g_\ell, g_r)\cdot g = g_\ell gg_r\inv$; we call the resulting action of $G_i\times G_i$ the {\em left-right action}.  Then $\mathbb{G}$ inherits a left-right action of $\mathbb{G}^2$, inducing an action on $T^*\mathbb{G}$ as well.  
For a reductive subgroup $G\subset\mathbb{G}^2$, we may choose a representation $V$ of $G$ and obtain a product $\mathsf{M} = T^*\mathbb{G}\times T^*V$ with an induced $G$-action.  As we explain below, $T^*\mathbb{G}$ admits a hyperk\"ahler metric, constructed by Kronheimer; if $K\subset G$ is a maximal compact subgroup, the Kronheimer metric is $K$-equivariant.

\subsection{Special Case}\label{sec:special case}
In Section \ref{main proof} below, we will consider the group 
\bd
G:=  GL_n(\mathbb{C}) \; \text{with maximal compact subgroup} \; K: = U(n);
\ed
$G$ acts on itself by the adjoint action, preserving the subgroup $SL_ n$.  The maximal compact subgroup $K$ thus acts compatibly on $SL_n$ and $GL_n$, with induced actions on their cotangent bundles. 
\begin{remark}
Henceforth, we $GL_n$-equivariantly identify $\mathfrak{gl}_n$ with $\mathfrak{gl}_n^*$ via the pairing $(a,b) \mapsto \on{Tr}(ab)$.
\end{remark}
Consider the vector representation $\mathbb{C}^n$ of $GL_n$, and the induced $GL_n$-action on $T^*\mathbb{C}^n$.  Write
\bd
\mathsf{M}_n:=T^*SL_n \times T^*\mathbb{C}^n=  SL_n\times \mathfrak{sl}_n^* \times T^*\mathbb{C}^n
\ed
with its induced $GL_n$-action.  

Writing
$\mathfrak{gl}_ n = \mathfrak{sl}_n\times \mathbb{C}$ as $GL_n$-representations, we obtain a closed immersion $T^*SL_n\hookrightarrow T^*GL_n$.  Via the trace identification, $T^*SL_n$ is identified with the fiber $(\on{det}\times \on{tr})\inv(1,0)$ of the map
\begin{equation}
GL_n\times \mathfrak{gl}_n\xrightarrow{(\on{det}\times \on{tr})} \Cs\times \mathbb{C}.
\end{equation}
One checks that the canonical complex moment map for the $GL_n$-action on $T^*(GL_n\times\mathbb{C}^n)$ is given, under the trace identification, by 
\bd
\overline{\mu}_{\mathbb{C}}(X,Y, i,j) = XYX\inv - Y + ij \hspace{1em} \text{for} \hspace{1em} (X,Y,i,j)\in GL_n\times\mathfrak{gl}_n\times \mathbb{C}^n\times(\mathbb{C}^n)^*.
\ed
Via the identification of $\mathfrak{gl}_n^*\cong \mathfrak{gl}_n$ and the resulting identification $\mathfrak{sl}_n^*\cong \mathfrak{sl}_n$, we find that the restriction of $\overline{\mu}_{\mathbb{C}}$ to $T^*SL_n\times T^*\mathbb{C}^n$ is identified with the complex moment map $\mu_{\mathbb{C}}$ for the latter.

\subsection{Hyperk\"ahler Structure}\label{sec:hk structure}
For this section we fix a complex reductive group $G$. 

We now consider the space $\mathcal{A} = C^\infty\big([0,1],\mathfrak{k}\otimes\mathbb{H}\big)$ of smooth maps from the interval $[0,1]$ to the quaternionic Lie algebra $\mathfrak{k}\otimes\mathbb{H}$.  Write $T = (T_0,T_1,T_2,T_3)$ for an element of $\mathcal{A}$.  The gauge group $\mathcal{G} = C^\infty_0\big([0,1],K\big)$ of smooth maps $f: [0,1] \rightarrow K$ that satisfy $f(0) = e = f(1)$, where $e\in K$ is the identity element, acts on $\mathcal{A}$ by
$f\cdot (T_0,T_1,T_2,T_3) = (fT_0f\inv + \frac{df}{dt}f\inv, fT_1f\inv, fT_2f\inv, fT_3f\inv)$.

As explained by Kronheimer \cite{Kronheimer} and subsequently explored and clearly exposed in \cite{DancerSwann}, the zero pre-image $Z$ of the natural associated infinite-dimensional hyperk\"ahler moment map consists of those $(T_0,T_1,T_2,T_3)$ satisfying {\em Nahm's equations},
\bd
\frac{dT_i}{dt} + [T_0,T_i] = [T_j,T_k],
\ed
where $(i,j,k)$ is a cylic permutation of $(1,2,3)$.   Kronheimer shows that $Z/\mathcal{G} \cong T^*G$.  It follows from the construction that $T^*G$ inherits a complete hyperk\"ahler metric that is $K\times K$-invariant under the left-right action.  

One can more easily see the complex structure $I$, the standard complex structure, on $T^*G$ from the {\em complex Nahm equation}.  More precisely, write $(\alpha,\beta) = (T_0 + iT_1, T_2+iT_3)$; the complex Nahm equation is 
$\frac{d\beta}{dt} + [\alpha,\beta] =0$.  Write $Z_{\mathbb{C}}\subset \mathcal{A}_{\mathbb{C}}= \{(\alpha,\beta)\}$ for its space of solutions, the zero preimage of a complex moment map for the action on $\mathcal{A}_{\mathbb{C}}$ of the complex group $\mathcal{G}_{\mathbb{C}} = C^\infty_0([0,1],G)$ via
\begin{equation}\label{complex gauge group action}
g(t)\cdot (\alpha(t),\beta(t)) = \big(\on{Ad}_{g(t)}(\alpha(t)) - \frac{dg}{dt} g(t)\inv, \on{Ad}_{g(t)}(\beta(t))\big).
\end{equation}
and one gets a diffeomorphism 
$Z/\mathcal{G} \rightarrow Z_{\mathbb{C}}/\mathcal{G}_{\mathbb{C}}$ defined by 
$(T_0,T_1,T_2,T_3)\mapsto (T_0-iT_1,T_2+iT_3)$.  

Under Kronheimer's identification of $Z_{\mathbb{C}}/\mathcal{G}_{\mathbb{C}}$ with $T^*G$, the left-right action of $e^Y \in G\times G$ for $Y\in\mathfrak{g}\times\mathfrak{g}$ takes on a particularly simple form.   Namely, write $Y = (Y_\ell, Y_r)$, and consider the functions $\mathsf{Y}: [0,1]\rightarrow \mathfrak{g}$ and $g: [0,1]\rightarrow G$ defined by 
\begin{equation}\label{Lie algebra function}
\mathsf{Y}(t) = (1-t)Y_\ell + tY_r \; \text{and} \; g(t) = e^{\mathsf{Y}(t)}.
\end{equation}
Then $g(t)$  naturally acts on $\mathcal{A} = \mathcal{A}_{\mathbb{C}}$ preserving $Z_{\mathbb{C}}$.  It follows from the discussion at the end of Section 2 of \cite{DancerSwann} that this action is identified with the left-right action of $e^Y$ on $T^*G$.  

\subsection{K\"ahler Potential}
It is known that there is a (global) K\"ahler potential $\mu_K$ for the Kronheimer metric on $T^*G$ for $G$ any complex reductive group.
  
More precisely, \cite{HKLR} show that if one forms the hyperk\"ahler reduction of a, possibly infinite-dimensional, hyperk\"ahler manifold by an appropriate group of hyperk\"ahler isometries; and if, in addition, the reduced manifold comes equipped with an $S^1$-action that rotates complex structures $I$ and $J$ and fixes complex structure $K$, then the moment map $\mu_K$ for this $S^1$-action is a global potential for the hyperk\"ahler metric.  

It is computed in \cite{DancerSwann} that $\mu_K(T_0,T_1,T_2,T_3) = \int_0^1 |T_1|^2 + |T_2|^2$.  Mayrand proves:
\begin{prop}[Mayrand \cite{Mayrand}]\label{Mayrand Kahler potential}
The K\"ahler potential $\mu_K$ is proper and bounded below.
\end{prop}
\noindent

In order to more easily analyze its growth farther on, we wish to modify the K\"ahler potential $\mu_K$ as follows.  Recall that the inner product on $\mathfrak{k}$ is given by $\langle a,b\rangle = \on{Tr}(a\overline{b}^{\on{t}})$.
We use repeatedly that for $T\in \mathfrak{u}(n)$, $\overline{T}^{\on{t}} = -T$.  Then
$\overline{\beta}^{\on{t}} = \overline{T}_2^{\on{t}} - i\overline{T}_3^{\on{t}} = -T_2 + iT_3$. 
Thus  $\beta - \overline{\beta}^{\on{t}} = 2T_2$.  Now 
\begin{equation}\label{mu calc}
\int_0^1 |T_2|^2 = \int_0^1 T_2\overline{T}_2^{\on{t}}
= \frac{1}{4}\int_0^1  (\beta-\overline{\beta}^{\on{t}})\overline{(\beta-\overline{\beta}^{\on{t}})}^{\on{t}}
= -\frac{1}{4}\int_0^1 \on{Tr}(\beta - \overline{\beta}^{\on{t}})^2.
\end{equation}
It is clear from the above description of complex structure $I$ that the function
\begin{equation}\label{defn of h}
h(\beta) = \frac{1}{4}\int_0^1 \on{Tr}\beta^2
\end{equation}
defines a holomorphic function on $T^*G$ (it is holomorphic by the above, and it is clearly invariant under the complex gauge group, hence descends to $T^*G$ in complex structure $I$).  
It follows:
\begin{lemma}
The function $\mu_K + h + \overline{h}$ is also a K\"ahler potential for the Kronheimer metric.  
\end{lemma}
We have 
\begin{equation}\label{hbar calc}
\overline{h(\beta)} = \frac{1}{4}\int_0^1 \on{Tr}\overline{\beta}^2 = \frac{1}{4}\int_0^1 \on{Tr}\big(\overline{\beta}^{\on{t}}\big)^2.
\end{equation}
Combining \eqref{mu calc} and \eqref{hbar calc} gives:
\bd
(\mu_K + h + \overline{h})(\alpha, \beta) = \int_0^1 |T_1|^2 + \frac{1}{4}\int_0^1 \on{Tr}\big(\beta  \overline{\beta}^{\on{t}} + \overline{\beta}^{\on{t}} \beta\big).
\ed
Using that $\on{Tr}(AB) = \on{Tr}(BA)$ and that $\overline{T}_i^{\on{t}} = - T_i$ for $T_i\in\mathfrak{k}$ gives 
\bd
\on{Tr}\beta \overline{\beta}^{\on{t}} = \on{Tr} \overline{\beta}^{\on{t}}\beta = \on{Tr}\left[(T_2 + i T_3) (\overline{T}_2^{\on{t}} - i \overline{T}_3^{\on{t}})\right]    = |T_2|^2 + |T_3|^2, \; \text{yielding} 
\ed
\begin{equation}\label{Kahler potential fmla}
(\mu_K + h + \overline{h})(\alpha, \beta)   = \int_0^1 |T_1|^2 + \frac{1}{2} |T^2_2| + \frac{1}{2} |T_3^2|
 = \int_0^1 |\on{Im}(\alpha)|^2 + \frac{1}{2} |\beta|^2.
\end{equation}
In particular, it is immediate that 
\begin{equation}\label{potential comparison}
\mu_K + h + \overline{h} \geq \frac{1}{2}\mu_K.
\end{equation}  
Recall that, under the identification of the moduli space of solutions of Nahm's equations with $T^*G$, the left-right action of $K\times K$ is identified with the action of $C^\infty([0,1],K)$.  It is immediate that:
\begin{lemma}\label{invariance}
The potential \eqref{Kahler potential fmla} is invariant under the left-right action of $K\times K$.
\end{lemma}

Suppose given a finite collection $G_i$, $i=1,\dots, n$ of complex reductive groups and a representation $V$ of $\mathbb{G} := \prod_i G_i$.  Choose a Hermitian inner product on $V$ that is $\mathbb{K}$-invariant for $\mathbb{K} = \prod_i K_i$, $K_i$ a maximal compact subgroup of $G_i$.  
We now equip $\mathsf{M} = \prod_i T^*G_i \times T^*V$ with the product hyperk\"ahler metric, where each $T^*G_i$ has the Kronheimer metric and $T^*V$ is given the flat metric with K\"ahler potential $|\cdot|^2$.  We thus obtain a K\"ahler potential $F_1(x,y) = \mu_K(x) + h(x) + \overline{h}(x) +|y|^2$ on $T^*\mathbb{G}\times T^*V$. The following is immediate from Proposition \ref{Mayrand Kahler potential} and the inequality \eqref{potential comparison}.
\begin{corollary}\label{proper-bdd-below}
The K\"ahler potential $F_1: \mathsf{M}\rightarrow \mathbb{R}$ is proper and bounded below.
\end{corollary}

\subsection{The GIT/Hyperk\"ahler Quotient Construction}
Give a finite collection $G_i$, $i=1,\dots, n$ of complex reductive groups as above, choose and fix a complex reductive subgroup $G\subset \mathbb{G}\times\mathbb{G}$, acting on $T^*\mathbb{G}$ via the left-right action; write $K\subset G$ for a choice of maximal compact subgroup.
Choose a representation $V$ of $G$, and let $\mathsf{M} = T^*(\mathbb{G}\times V)$.  Choose a character $\chi: G\rightarrow \Gm$ and write $\lambda = d\chi:\mathfrak{g}\rightarrow\mathbb{C}$.
\begin{thm}\label{kempf-ness thm}
The map 
\bd
\mu_{\mathbb{C}}\inv(0)\cap \mu_{\mathbb{R}}\inv(-i\lambda) \hookrightarrow \mu_{\mathbb{C}}\inv(0)
\ed
 takes values in $\big(\mu_{\mathbb{C}}\inv(0)\big)^{\on{\chi-ss}}$ and induces a homeomorphism of topological stacks
\bd
(\mu_{\mathbb{C}}\inv(0)\cap \mu_{\mathbb{R}}\inv(-i\lambda))/K \simeq \mu_{\mathbb{C}}\inv(0)^{\chi-\on{ss}}/G
\ed
and a
 complex-analytic isomorphism
\bd
 \big(\mu_{\mathbb{C}}\inv(0)\cap \mu_{\mathbb{R}}\inv(-i\lambda)\big)/K  = \mathsf{M}/\!\!/\!\!/_{(0,-i\lambda)} K \xrightarrow{\cong} \mu_{\mathbb{C}}\inv(0)/\!\!/_{\chi} G.
\ed
\end{thm}
\begin{remark}
As in the introduction, we remark that the (closely related) theorem was known to experts (cf. Sections 2.2 and 3.3 of \cite{NT}). 
Also as remarked in the introduction, the proof follows the usual approach: standard methods and results (see \cite{Sjamaar}) immediately reduce the proof to showing a certain assertion about limits of downward Morse flows, which is the main content of Section \ref{kempf-ness proof} and which does not seem to have been previously documented in the literature.
\end{remark}

It follows that we get a commutative diagram of equivariant cohomology groups
\begin{equation}\label{comm diagram}
\xymatrix{
H^*_G\big(\mathsf{M},\mathbb{Q}) \ar[d]\ar[r]^{\cong} &  H^*_K(\mathsf{M},\mathbb{Q})\ar[d] \\
H^*_G\big(\mu_{\mathbb{C}}\inv(0)^{\chi-\on{ss}}\big) \ar[r]^{\hspace{-1em}\cong} &  H^*_K(\mu_{\mathbb{C}}\inv(0)\cap \mu_{\mathbb{R}}\inv(\xi),\mathbb{Q}).
}
\end{equation}
We use the identifications in diagram \eqref{comm diagram} to reduce Theorem \ref{main thm} to the corresponding assertion about the algebraic-symplectic Kirwan map, which we analyze in Section \ref{main proof}.  The proof of Theorem \ref{kempf-ness thm} appears in Section \ref{kempf-ness proof}.

\section{Proof of Theorem \ref{kempf-ness thm}}\label{kempf-ness proof}
As explained above, we equip $\mathsf{M} = T^*(\mathbb{G}\times V)$ with the product of the Kronheimer hyperk\"ahler metrics on $T^*\mathbb{G}$ and the flat hyperk\"ahler metric on $T^*\mathbb{C}^n$.  

\subsection{K\"ahler Potential and Moment Map Pre-Image}
As above, we write $F_1: T^*\mathbb{G}\times T^*V\rightarrow \mathbb{R}$ for the K\"ahler potential on $\mathsf{M}$ given by $F_1(x,y) = \mu_K(x) + h(x) + \overline{h}(x) +|y|^2$.  We further equip $\mathbb{C}$ with the singular K\"ahler metric (the ``lifted Fubini-Study metric'') with K\"ahler potential $F_2(z) =  \frac{1}{2}\log |z|^2$.  We define 
$F: T^*\mathbb{G}\times T^*V\times \mathbb{C}\rightarrow \mathbb{R}$, 
\begin{equation}\label{Kahler potential}
F(x,y,z) = F_1(x,y) + F_2(z) = \big(\mu_K(x)  + h(x) + \overline{h}(x)\big) + |y|^2 + \frac{1}{2}\log|z|^2.
\end{equation}

We let $G\subset \mathbb{G}\times\mathbb{G}$ act on $\mathbb{C}$ by 
$g\cdot z = \chi(g)z$, and give $T^*\mathbb{G}\times T^*V \times \mathbb{C}$ the product action of $G$.   Then $K$ acts preserving the K\"ahler metric, and the K\"ahler potentials $F_1$ and  $F$ are evidently $K$-invariant. 
As above, write $\lambda = d\chi: \mathfrak{g}\rightarrow\mathbb{C}$.

Choose $(x,y,z) \in T^*\mathbb{G}\times T^*V\times (\mathbb{C}\smallsetminus\{0\})$, and let $\theo = G\cdot (x,y,z)$ be its $G$-orbit.  We use:
\begin{thm}[\cite{Mostow}, Theorem~4.1]\label{mostow}
Let $H\subset K$ be a closed subgroup of a compact connected Lie group $K$.  Then there is an $H$-invariant subspace $\mathfrak{m}\subset\mathfrak{k}$ for which the map $(k,v)\mapsto ke^{iv}$ induces a diffeomorphism $K\times_H \mathfrak{m}\rightarrow K_{\mathbb{C}}/H_{\mathbb{C}}$.
\end{thm}
Assume that the orbit $\theo$ is closed in $T^*\mathbb{G}\times T^*V\times\mathbb{C}$: in other words, that the point $(x,y)\in T^*\mathbb{G}\times T^*V$ is a $\chi$-semistable point of 
$T^*\mathbb{G}\times T^*V$ under the action of $G$.
Then the orbit $\theo$ is affine, and it follows that the stabilizer $G_{(x,y,z)}$ is a reductive subgroup of $G$.  Choosing an appropriate $(x,y,z)\in \theo$, we may assume chosen, without loss of generality, a maximal compact subgroup $H$ of $G_{(x,y,z)}$ that is also a subgroup of $K$.  Mostow's Theorem \ref{mostow} yields a diffeomorphism
\begin{equation}\label{coords on O}
K\times_H \mathfrak{m} \xrightarrow{\cong} \theo, \hspace{2em} (k,v)\mapsto ke^{iv}\cdot (x,y,z).
\end{equation}

For $X\in\mathfrak{g}$, write $\wt{X}$ for the induced vector field on $\theo$.  
\begin{lemma}\label{moment map lemma}
For any $q = g\cdot (x,y)\in \mathsf{M}$, $p = g\cdot (x,y,z)\in \theo$ and $X\in\mathfrak{k}$, we have 
\bd
dF_p(\wt{iX}) = \langle \mu_{\mathbb{R}}(q),X\rangle + \lambda(iX).
\ed
\end{lemma}
\begin{proof}
Using $dF = dF_1 + \frac{1}{2}d\log |z|^2$, we get $(dF_1)_q(\wt{iX}) = \langle \mu_{\mathbb{R}}(q),X\rangle$ from \cite[Proposition~4.1]{Mayrand}.  Now observe that since $X\in\mathfrak{k}$, we have $\chi(e^{isX}) = e^{\lambda(iX)s}$ where $\lambda(iX)\in\mathbb{R}$.  Then
 $\frac{d}{ds}\log|\chi(e^{isX})|^2|_{s=0} = 2\lambda(iX)$, yielding the desired result.
\end{proof}
\begin{prop}\label{critical points prop}
\mbox{}
\begin{enumerate}
\item A point $p\in\theo$ is a critical point of $F|_{\theo}$ if and only if $p\in \mu_{\mathbb{R}}\inv(-i\lambda)$.
\item Any critical point of $F|_{\theo}$ is a minimum. 
\item The critical locus of $F|_{\theo}$ is either empty, or consists of a single $K$-orbit. 
\end{enumerate}
\end{prop}
\begin{proof}
Assertion (1) is immediate from Lemma \ref{moment map lemma}. 
Assertion (2) is immediate from Lemma 1 of \cite{AL}.  Assertion (3) follows from the Proposition of \cite{AL}.
\end{proof}

\begin{prop}\label{quotient diffeo}
Suppose that for every closed orbit $\theo = G\cdot (x,y,z)$ for $z\neq 0$, the function $F|_{\theo}$ is proper and bounded below.  Then the map $\mu_{\mathbb{C}}\inv(0)\cap \mu_{\mathbb{R}}\inv(-i\lambda) \rightarrow \mu_{\mathbb{C}}\inv(0)^{\on{det-ss}}$ induces a diffeomorphism
\begin{equation}\label{eq:orbit spaces}
\big(\mu_{\mathbb{C}}\inv(0)\cap \mu_{\mathbb{R}}\inv(-i\lambda)\big)/K 
\longrightarrow
\mu_{\mathbb{C}}\inv(0)/\!\!/_\chi G.
\end{equation}
\end{prop}
\begin{proof}
It follows from the hypothesis that each closed orbit $G\cdot (x,y,1)$ of $T^*\mathbb{G}\times T^*V\times\mathbb{C}$ contains a critical point.  Proposition \ref{critical points prop}(3) then implies that each closed orbit contains a unique $K$-orbit of critical points.  It thus follows by Proposition \ref{critical points prop}(1) that a $\chi$-semistable orbit $G\cdot (x,y)\subset T^*\mathfrak{G}\times T^*V$ contains a unique $K$-orbit in $\mu_{\mathbb{R}}\inv(-i\lambda)$.  The map \eqref{eq:orbit spaces} is thus a $C^\infty$ bijection.  That it is a diffeomorphism is immediate by standard arguments. 
\end{proof}

Unfortunately, we do not know a general result that would establish the hypothesis of Proposition \ref{quotient diffeo} without some specific information about the K\"ahler potential $F$---though one might hope that a general result should exist.
Thus, in Section \ref{properness subsection}, we prove by an explicit analysis of the growth of the K\"ahler potential that when $\theo$ is a closed orbit, then $F|_{\theo}$ is indeed proper and bounded below.

\subsection{Action of Semisimple Elements of $\mathfrak{g}$}
We will say that a function $G(s)$ of a real variable $s\geq 0$ {\em grows exponentially in $s$} if there exist constants $c>0, C>0$, and $C_0$ such that 
$G(s)\geq Ce^{cs} + C_0$ for all $s\geq 0$.
We say $G(s)$ {\em grows quadratically in $s$} if there is a real polynomial $a_2s^2 + a_1s + a_0$ with $a_2>0$ such that $G(s)\geq a_2s^2 + a_1s + a_0$ for all $s\geq 0$.

  Suppose that $V$ is a finite-dimensional complex vector space with Hermitian inner product $\langle\cdot, \cdot\rangle$, and that $X\in\End(V)$ is a semisimple endomorphism that is skew-Hermitian (so $\langle Xv,w\rangle + \langle v,Xw\rangle = 0$ for all $v,w\in V$); then $X$ has imaginary eigenvalues.  
  Since $X$ is semisimple and skew-Hermitian, $iX$ is semisimple with real eigenvalues.   
  Decompose $V=\oplus_\eta  V_\eta$, a direct sum of eigenspaces for $X$ with eigenvalue $\eta$.  Furthermore,
the subspaces $V_\eta, V_{\eta'}$ are  orthogonal for distinct $\eta,\eta'$, since they are orthogonal for $X$.
   Then:
\begin{lemma}\label{splitting lemma}
For any $v = \sum v_\eta \in V$, with $v_\eta\in V_\eta$, we have, for $s\in\mathbb{R}$,
\bd
e^{isX}\cdot v = \sum_\eta e^{s\eta} v_\eta \; \text{and} \; |e^{isX}\cdot v|^2 = \sum_\eta |e^{s\eta}v_\eta|^2 
= \sum_\eta e^{2s\eta}|v_\eta|^2.
\ed
\end{lemma}
Continuing with the above hypotheses, suppose next that $V$ comes equipped with a real structure, i.e., a complex antilinear involution $\sigma: V\rightarrow V$.  
The imaginary part of a vector $v$ is then $\on{Im}(v) = \frac{v+\sigma(v)}{2i}$.  
If $X$ is a real endomorphism, so 
$\sigma\circ X = X\circ\sigma$,
and if $v\in V_\eta$ is an eigenvector for $iX$ with eigenvalue $\eta$, then $iX\sigma(v) = \sigma(-iXv) = \sigma(-\eta v) = -\eta\sigma(v)$, so $\sigma(v)\in V_{-\eta}$. 

\subsection{Action of Certain Gauge Transformations}
 Fixing $X_\ell, X_r \in \mathfrak{u}(V)$ and taking $Y_\ell = iX_\ell$, $Y_r = iX_r$ and defining $\mathsf{Y}(t)$ as in Formula \eqref{Lie algebra function}, there is a finite subset $S = \{s_0,s_1,\dots,s_d\}\subset [0,1]$ so that on each interval $(s_{k-1}, s_k)$, the number of positive eigenvalues, and the multiplicity of each, of $\mathsf{Y}(t)$ is constant.   
  It follows from Proposition \ref{continuous diagonalization} that we can find an orthogonal eigenbasis for $\mathsf{Y}(t)$,
$v_i(t): \mathbb{R}\rightarrow V$, $i=1,\dots, n$, varying continuously in $t\in\mathbb{R}$, with continuously varying eigenvalues $\eta_i(t)$.
It follows that, fixing an interval $(s_{k-1}, s_k)$, there are index sets $J_+, J_-$ and continuous functions $\nu_j(t)$, $j\in J_+\sqcup J_-$ so that if $j\neq j'$ then $\nu_j(t)\neq \nu_{j'}(t)$ for all $t\in (s_{k-1}, s_k)$, and for each $i\in I_+$, respectively $i\in I_-$, there exists a unique $j\in J_+$, respectively a  unique $j\in J_-$, with $\eta_i(t) = \nu_j(t)$ on $[s_{k-1}, s_k]$.  We assume the index sets $J_+, J_-$ have been chosen so that $\eta_{-j}(t) = -\eta_j(t)$ for all $j$.

A continuous map $v(t):[s_{k-1},s_k]\rightarrow V$ then admits a unique expression
$v(t) = \sum_{j\in J_+} v_j(t) + v_0 + \sum_{j\in J_-} v_j(t)$, with $\mathsf{Y}(t)v_j(t) = \nu_j(t)v_j(t)$ for $j\in J_+\sqcup J_-$ and $\mathsf{Y}(t)v_0(t) = 0$.

\begin{lemma}\label{imaginary estimate}
For any continuous $v(t):[s_{k-1},s_k]\rightarrow V$, 
either
$\ds\int_{s_{k-1}}^{s_k}|\on{Im}(e^{s\mathsf{Y}(t)}\cdot v(t))|^2\, dt$  is bounded as a function of $s\geq 0$ or it grows exponentially in $s$.
\end{lemma}
\begin{proof}
We have
\begin{multline*}
\on{Im}(e^{s\mathsf{Y}(t)}v(t)) =
\on{Im}(v_0(t)) + \frac{1}{2i}\sum_{j\in J_+}\big[e^{s\eta_j(t)}v_j(t) + \sigma(e^{s\eta_j(t)}v_j(t)) + e^{-s\eta_j(t)}v_{-j}(t)+ \sigma\big(e^{-s\eta_j(t)}\sigma(v_{-j}(t))\big)\big] \\
=
 \on{Im}(v_0(t)) + \frac{1}{2i}\sum_{j\in J_+}\big[(e^{s\eta_j(t)}(v_j(t) + \sigma(v_j(t))) + e^{-s\eta_j(t)}(v_{-j}(t) + \sigma(v_{-j}(t)))\big].
\end{multline*}
Since $ \on{Im}(v_0(t))$ and $\sum_{j\in J_+} e^{-s\eta_j(t)}\big(v_{-j}(t) + \sigma(v_{-j}(t))\big)$ are continuous functions of $t$, uniformly bounded in norm for $s\geq 0$, we find that $\int_{s_{k-1}}^{s_k}|\on{Im}(e^{s\mathsf{Y}(t)}\cdot v)|^2\, dt$ is unbounded as a function of $s\geq 0$ if  and only if
\begin{equation}\label{exponential estimation}
\frac{1}{4}\int_{s_{k-1}}^{s_k}\big|\sum_{j\in J_+}(e^{s\eta_j(t)}(v_j(t) + \sigma(v_j(t))) \big|^2 \, dt 
= \frac{1}{4}
\sum_{j\in J_+}
\int_{s_{k-1}}^{s_k}e^{2s\eta_j(t)}\big[|v_j(t)|^2 + |\sigma(v_j(t)|^2\big]\, dt
\end{equation}
is unbounded as a function of $s$.  This happens if and only if $|v_j(t)|^2 + |\sigma(v_j(t)|^2>0$ for some $j$ and some $t\in [s_{k-1}, s_k]$: in that case, by continuity we can find some closed interval $[a,b]\subset (s_{k-1}, s_k)$ (with $b>a$) and $\epsilon >0$ for which $|v_j(t)|^2 + |\sigma(v_j(t)|^2>\epsilon$ on $[a,b]$; and then there exists a $C>0$ with $\eta_j(t)>C$ on $[a,b]$, showing that the right-hand side of \eqref{exponential estimation} is bounded below by $\frac{(b-a)\epsilon}{4} e^{2Cs}$, which grows exponentially in $s$.
\end{proof}

Note that, taking $g(t) = e^{s\mathsf{Y}(t)}$, we get $\frac{dg}{dt}g(t)\inv = s\frac{d\mathsf{Y}}{dt} = s(Y_r - Y_\ell)$.  
The modified K\"ahler potential $\mu_K + h + \overline{h}$ of \eqref{Kahler potential fmla}  thus satisfies
\bd
\ds(\mu_K+h+\overline{h})\big(\on{Ad}_{e^{s\mathsf{Y}(t)}}(\alpha,\beta)\big)  = 
 \int_0^1|\on{Im}(\on{Ad}_{e^{s\mathsf{Y}(t)}}\alpha(t)) +s\on{Im}(Y_r - Y_\ell)|^2\, dt + \frac{1}{2}\int_0^1|\on{Ad}_{e^{s\mathsf{Y}(t)}}\beta(t)|^2 \,dt.
\ed

\begin{prop}\label{estimate}
Let $(Y_\ell, Y_r)\in\mathfrak{g}\times\mathfrak{g}$ and let $\mathsf{Y}(t) = (1-t)Y_\ell + tY_r$.  If
$(\mu_K + h + \overline{h})(\on{Ad}_{e^{s\mathsf{Y}(t)}}(\alpha,\beta))$ is unbounded as a function of $s$ then $(\mu_K + h + \overline{h})(\on{Ad}_{e^{s\mathsf{Y}(t)}}(\alpha,\beta))$ grows at least quadratically in $s$.
\end{prop}
\begin{proof}
Suppose first that the term $\ds\frac{1}{2}\int_0^1|\on{Ad}_{e^{s\mathsf{Y}(t)}}\beta(t)|^2 \,dt$ is unbounded as a function of $s$.  Combining Proposition \ref{continuous diagonalization} with Lemma \ref{splitting lemma}, we get
\bd
\frac{1}{2}\int_0^1|\on{Ad}_{e^{s\mathsf{Y}(t)}}\beta(t)|^2 \,dt  = 
\frac{1}{2}\sum_i \int_0^1 e^{2s\eta_i(t)}|\beta_i(t)|^2\, dt.
\ed  Either each $\beta_i(t)\leq 0$ whenever $\eta_i(t)>0$, in which case the above expression is bounded as a function of $s$; or there is some closed interval $[a,b]$ and $\epsilon>0, C>0$ on which $|\beta_i(t)|>\epsilon$ and $\eta_i(t)>0$, showing that the above expression is bounded below by $\frac{(b-a)\epsilon}{2}e^{2Cs}$ and thus grows exponentially in $s$.   We conclude that $\ds(\mu_K+h+\overline{h})\big(\on{Ad}_{e^{s\mathsf{Y}(t)}}(\alpha,\beta)\big)$ grows exponentially as a function of $s$.  

Suppose next that $\ds\int_0^1|\on{Im}(\on{Ad}_{e^{s\mathsf{Y}(t)}}\alpha(t))|^2 \, dt$ is unbounded as a function of $s$.  
Applying Lemma \ref{imaginary estimate} to $\ds\int_0^1|\on{Im}(\on{Ad}_{e^{s\mathsf{Y}(t)}}\alpha(t))|^2 \, dt$, we see that either it is bounded in $s$ or grows exponentially in $s$.  Since $|s\on{Im}(Y_r-Y_\ell)|^2$ is bounded above by a quadratic polynomial in $s$, if $\ds\int_0^1|\on{Im}(\on{Ad}_{e^{s\mathsf{Y}(t)}}\alpha(t))|^2 \, dt$ grows exponentially in $s$ then $\ds\int_0^1|\on{Im}(\on{Ad}_{e^{s\mathsf{Y}(t)}}\alpha(t)) +s\on{Im}(Y_\ell - Y_r)|^2\, dt$ grows exponentially in $s$, implying that $\ds(\mu_K+h+\overline{h})\big(\on{Ad}_{e^{s\mathsf{Y}(t)}}(\alpha,\beta)\big)$ grows exponentially in $s$.

Finally, suppose that $\ds\int_0^1|\on{Im}(\on{Ad}_{e^{s\mathsf{Y}(t)}}\alpha(t))|^2 \, dt$ is bounded as a function of $s$.
Then either $\on{Im}(Y_r-Y_\ell) = 0$, in which case the entire integral $\ds\int_0^1|\on{Im}(\on{Ad}_{e^{s\mathsf{Y}(t)}}\alpha(t)) +s\on{Im}(Y_r - Y_\ell)|^2\, dt$ is bounded as a function of $s$, or $\on{Im}(Y_r-Y_\ell) \neq 0$, in which case
the entire integral $\ds\int_0^1|\on{Im}(\on{Ad}_{e^{s\mathsf{Y}(t)}}\alpha(t)) +s\on{Im}(Y_r - Y_\ell)|^2\, dt$ grows quadratically in $s$.  

Combining the conclusions of the previous three paragraphs yields the conclusion.
\end{proof}

\begin{remark}\label{uniformity remark}
The proofs of Lemma \ref{imaginary estimate} and Proposition \ref{estimate} show the following. Let $X(z)$ be an element of $\mathfrak{u}(V)$ depending continuously on $z\in B$ where $B$ is a small ball around $0\in\mathbb{R}^N$, and associate the function $\mathsf{Y}(t,z)$ via Formula \eqref{Lie algebra function}.  
If 
$(\mu_K + h + \overline{h})(\on{Ad}_{e^{s\mathsf{Y}(t,0)}}(\alpha,\beta))$
is unbounded as a function of $s$, then there are a small ball $B$ around $0\in\mathbb{R}^N$ and 
a choice of $a_2s^2 + a_1s + a_0$ with $a_2>0$ such that 
\bd
(\mu_K + h + \overline{h})(\on{Ad}_{e^{s\mathsf{Y}(t,z)}}(\alpha,\beta)) \geq 
a_2s^2 + a_1s + a_0
\ed
for all $s\geq 0$ and $z\in B$.
\end{remark}

\subsection{Properness of $F|_\theo$}\label{properness subsection}
We now are ready to prove:
\begin{prop}\label{bbp}
For each closed orbit $\theo = G\cdot (x,y,z)\subset T^*\mathbb{G}\times T^*V\times\mathbb{C}$ with $z\neq 0$, the function $F|_{\theo}$ is proper and bounded below.
\end{prop}

We will make extensive use of Mostow's coordinates \eqref{coords on O} on $\theo$.  Since $F$ is $K$-invariant, it suffices to prove that the composite $F\circ \phi$, with $\phi: \mathfrak{m}\rightarrow\theo$ defined by $\phi(X)= e^{iX}\cdot(x,y,z)$, is proper and bounded below.  
\begin{assumption}
We assume without loss of generality that $z=1$.
\end{assumption}
Note that $\log|\chi(e^{isX})|^2 = 2\lambda(iX)s$ with $\lambda(iX)\in\mathbb{R}$.  
Then:
\begin{prop}\label{growth estimates prop}
For each $X\in\mathfrak{m}$, the function $G(s)= F\big(e^{isX}\cdot (x,y,z)\big)$ satisfies one of:
\begin{enumerate}
\item $\lambda(iX)> 0$ and $F\big(e^{isX}\cdot (x,y,z)\big)\geq F_1\big(e^{isX}\cdot (x,y)\big)$  for all $s\geq 0$.
\item $\lambda(iX)\leq 0$ and $F\big(e^{isX}\cdot (x,y,z)\big)$ grows at least quadratically in $s$.
\end{enumerate}
\end{prop}
\begin{proof}
We consider the cases separately:

\vspace{.6em}

\noindent
{\em Case 1.}  $\lambda(iX)\geq 0$.  Then $G(s) =  F_1(e^{isX}\cdot(x,y)) + \frac{1}{2}\log|\chi(e^{isX})|^2\geq F_1(e^{isX}\cdot(x,y))$.  

\vspace{.6em}

\noindent
{\em Case 2.} $\lambda(iX)\leq 0$.  If $F_1(e^{isX}\cdot (x,y))$ is bounded as a function of $s$, then by Corollary \ref{proper-bdd-below} the trajectory $\{e^{isX}\cdot (x,y)\}$ lies in a compact subset of $T^*\mathbb{G}\times T^*V$.  Thus, there exists an unbounded sequence $s_n$ in $\mathbb{R}_{\geq 0}$ for which the sequence $e^{is_nX}\cdot (x,y)$ converges in $T^*\mathbb{G}\times T^*V$, say to $(x_0,y_0)$.  

If $\lambda(iX)<0$,
it follows from the previous paragraph that $\ds\lim_{n\rightarrow\infty} e^{is_nX}\cdot (x,y,1) = (x_0,y_0,0)$ lies in the closure of $\theo$ in $T^*\mathbb{G}\times T^*V\times\mathbb{C}$; since its third coordinate is $0$ it cannot lie in $\theo$, a contradiction since $\theo$ was assumed closed.  Thus $F_1(e^{isX}\cdot (x,y))$ is unbounded as a function of $s$.  On the other hand, if
 $\lambda(iX) = 0$, then $\ds\lim_{n\rightarrow\infty} e^{is_nX}\cdot (x,y,1) = (x_0,y_0,1)$; since the orbit $\theo$ is assumed closed, we have $(x_0,y_0,1) = ke^{iX'}\cdot(x,y,1)$ for some $k\in K$ and $X'\in\mathfrak{m}$.  But Theorem \ref{mostow} then implies that $(1,s_nX)\rightarrow (k,X')$ in $K\times_L\mathfrak{m}$ as $n\rightarrow \infty$, which is obviously false.  Thus again $F_1(e^{isX}\cdot (x,y))$ is unbounded as a function of $s$.

Now choose a solution $(\alpha,\beta)$ of the complex Nahm equation representing $x\in T^*\mathbb{G}$.  By the conclusion of the previous paragraph, either $(\mu_K + h + \overline{h})(\on{Ad}_{e^{isX}}(\alpha,\beta))$ is unbounded as a function of $s$, or $|e^{isX}y|^2$ is unbounded as a function of $s$.  In the first case, identifying $iX = (Y_{\ell}, Y_{r})$ with $\mathsf{Y}(t)$ in the gauge group via Formula \eqref{Lie algebra function}, we conclude from Proposition \ref{estimate} that $(\mu_K + h + \overline{h})(\on{Ad}_{e^{isX}}(\alpha,\beta))$ grows at least quadratically as a function of $s$; while in the second case, we conclude 
from Lemma \ref{splitting lemma} that  $|e^{isX}y|^2$ grows exponentially as a function of $s$.  In either case, adding the linear function $\lambda(iX)s$ we still obtain that $F(e^{isX}\cdot(x,y,z))$ grows at least quadratically as a function of $s$.
\end{proof}

\begin{proof}[Proof of Proposition \ref{bbp}]
Let $S\subset\mathfrak{m}$ denote the unit sphere in $\mathfrak{m}$ in the $K$-invariant inner product induced from $\mathfrak{k}$.  We then get a proper surjective map $\mathbb{R}_{\geq 0}\times S \rightarrow \mathfrak{m}$, $(t,X)\mapsto tX$; it suffices to show that the composite map 
$\mathbb{R}_{\geq 0}\times S \rightarrow \mathbb{R}$, $(s,X) \mapsto F(e^{isX}\cdot (x,y,z))$ is proper and bounded below.

Consider the function $e:S\rightarrow \mathbb{R}$ defined by 
$e(X) = \lambda(iX)$; write 
\bd
S_+ = e\inv(\mathbb{R}_{>0}), \; S_0 = e\inv(0), \; S_- = e\inv(\mathbb{R}_{<0}).
\ed
By Proposition \ref{growth estimates prop} and Remark \ref{uniformity remark}, for every point $X\in S_-\sqcup S_0$ there are an open neighborhood $U_X$ of $X$ in $S_-\sqcup S_0$ and $a_2s^2 + a_1s + a_0$ with $a_2>0$ such that 
$F(e^{isX'}\cdot(x,y,z))\geq a_2s^2 + a_1s + a_0$ for all $X'\in U_X$.  Since $S_-\sqcup S_0$ is compact, it follows that there exists a single choice of $a_2s^2 + a_1s + a_0$ such that
$F(e^{isX'}\cdot(x,y,z))\geq a_2s^2 + a_1s + a_0$ for all $X'\in S_-\sqcup S_0$.  Thus the restriction of 
$(s,X)\mapsto F(e^{isX}\cdot (x,y,z))$
to $\mathbb{R}_{\geq 0}\times (S_-\sqcup S_0)$ is proper and bounded below.

Now, Proposition \ref{growth estimates prop} and Remark \ref{uniformity remark} together imply that the restriction of 
$(s,X)\mapsto F(e^{isX}\cdot (x,y,z))$
to a neighborhood of $\mathbb{R}_{\geq 0}\times S_0$ in $\mathbb{R}_{\geq 0}\times (S_+\sqcup S_0)$ grows at least quadratically in $s$; moreover, the proposition immediately implies that $(s,X)\mapsto F(e^{isX}\cdot (x,y,z))$ grows at least linearly in $s$, with a lower bound on the slope, on the complement of that neighborhood in $\mathbb{R}_{\geq 0}\times (S_+\sqcup S_0)$.  It follows that 
the restriction of 
$(s,X)\mapsto F(e^{isX}\cdot (x,y,z))$
to $\mathbb{R}_{\geq 0}\times (S_+\sqcup S_0)$
is proper and bounded below.

Combining the conclusions of the previous two paragraphs yields Proposition \ref{bbp}.\end{proof}

\begin{proof}[Proof of Theorem \ref{kempf-ness thm}]
The hypothesis of Proposition \ref{quotient diffeo} is supplied by Proposition \ref{bbp}.  Then Proposition \ref{quotient diffeo} immediately yields the conclusion.
\end{proof}

\section{Hilbert Schemes and Subvarieties}\label{main proof}
We now turn to the situation of Theorem \ref{main thm} of the introduction.  Thus, we return to the notation of Section \ref{sec:special case}.  

Consider $\mathsf{M} = \mathsf{M}_n = T^*SL_n\times T^*\mathbb{C}^n$.  
Applying Theorem \ref{kempf-ness thm} to the $GL_n = SL_n\times_{\mu_n} \Gm$-action induced from the adjoint action on $T^*SL_n$ and the obvious action on $\mathbb{C}^n$ shows that, for $\chi = \det: GL_n\rightarrow \Gm$ and $\xi = -i\det$, we have 
$\mu_{\mathbb{C}}\inv(0)/\!\!/_{\chi} G \cong \mathsf{M}/\!\!/\!\!/_{(0,\xi)} K$.  It follows that the hyperk\"ahler Kirwan map is identified with the map
\bd
H^*_{GL_n}(SL_n) \cong H^*_{GL_n}(T^*SL_n\times T^*\mathbb{C}^n) \xrightarrow{\kappa} H^*_{GL_n}\big(\mu\inv(0)^{\on{det-ss}}\big) = H^*(\mu\inv(0)/\!\!/_{\on{det}} GL_n).
\ed

As in Section \ref{sec:special case} above, the image of the natural embedding $T^*SL_n\times T^*\mathbb{C}^n\hookrightarrow T^*GL_n\times T^*\mathbb{C}^n$ is the preimage of $(1,0)$ under the natural map
\bd
T^*GL_n\times T^*\mathbb{C}^n \cong GL_n\times \mathfrak{gl}_n \times \mathbb{C}^n\times (\mathbb{C}^n)^* \ni (X,Y,i,j) \mapsto (\det(X),\on{tr}(Y))\in \mathbb{C}^*\times\mathbb{C}.
\ed
We write $\det\times\on{tr}$ for the map.  The map is clearly $GL_n$-invariant and thus descends to a map $\overline{\mu}_{\mathbb{C}}\inv(0)/\!\!/_{\on{det}} GL_n$.
\begin{lemma}\label{Hilb scheme}
The Hamiltonian reduction $\overline{\mu}_{\mathbb{C}}\inv(0)/\!\!/_{\on{det}} GL_n$ of $T^*GL_n\times T^*\mathbb{C}^n$ is isomorphic to the Hilbert scheme of points $(\mathbb{C}^*\times \mathbb{C})^{[n]}$.  Under this isomorphism, the function $\det\times\on{tr}$ is identified with the product, respectively sum, of the coordinates of the $n$ points.
\end{lemma}
\begin{proof}
The subset $\overline{\mu}_{\mathbb{C}}\inv(0)$ consists of $(X,Y,i,j)$ with $X$ invertible and 
$XYX\inv - Y + ij = 0$ is identified with the set of $(X,Y,i,j)\in T^*(\mathfrak{gl}_n\times\mathbb{C}^n)$ satisfying $[X,Y] + ij'=0$ and with $X$ invertible via $(X,Y,i,j)\mapsto (X,Y,i,j') = (X,Y,i, jX)$.  One easily sees that $\on{det}$-stability corresponds.  The result is thus immediate from \cite{Nakajima}.
\end{proof}

We obtain a commutative diagram
\begin{equation}\label{diagram}
\xymatrix{
H^*_{GL_n}(GL_n, \mathbb{Q})\ar[r] \ar[d] &  H^*_{GL_n}(SL_n, \mathbb{Q})\ar[d]_{\kappa}\\
H^*\big((\mathbb{C}^*\times\mathbb{C})^{[n]},\mathbb{Q}\big)\ar[r] &  H^*(\mathsf{M}/\!\!/\!\!/ K, \mathbb{Q}).
}
\end{equation}
For later reference, we note one easy topological fact.   Consider the natural map $H^*_{GL_n}(GL_n, \mathbb{Q})\rightarrow H^*_{GL_n}(SL_n, \mathbb{Q})$ of equivariant cohomology groups associated to the adjoint-equivariant inclusion $SL_n\rightarrow GL_n$.
\begin{prop}
The homomorphism of $\on{Ad}$-equivariant cohomology
\bd
H^*_{GL_n}(GL_n, \mathbb{Q})\rightarrow H^*_{GL_n}(SL_n, \mathbb{Q}) 
\ed
is surjective.
\end{prop}
\begin{proof}
We use the map $SL_n\xrightarrow{p} PGL_n$.  Over $\mathbb{C}$, $H^*(PGL_n, \mathbb{C})\rightarrow H^*(SL_n,\mathbb{C})$ is an isomorphism since both are identified with the cohomology of their common Lie algebra; hence $H^*(PGL_n,\mathbb{Q})\rightarrow H^*(SL_n,\mathbb{Q})$ is also an isomorphism.  This yields an isomorphism of $E_2$ pages for the Leray spectral sequences abutting to $p^*: H^*_{GL_n}(PGL_n,\mathbb{Q})\rightarrow H^*_{GL_n}(SL_n,\mathbb{Q})$, showing that $p^*$ is an isomorphism.  Since $p^*$ factors through $H^*_{GL_n}(GL_n,\mathbb{Q})\rightarrow H^*_{GL_n}(SL_n,\mathbb{Q})$, the conclusion follows.
\end{proof}

The right-hand vertical map $\kappa$ in \eqref{diagram} is the hyperk\"ahler Kirwan map \eqref{eq:HKK map} for our manifold $\mathsf{M}$.  Since the top horizontal map is surjective, if the hyperk\"ahler Kirwan map for $\mathsf{M}$ were surjective then the bottom horizontal arrow would be surjective.  We will show that the map 
\bd
H^*\big((\mathbb{C}^*\times\mathbb{C})^{[n]},\mathbb{Q}\big)\longrightarrow  H^*(\mathsf{M}/\!\!/\!\!/ K, \mathbb{Q})
= 
H^*\big((\on{det}\times\on{tr})\inv(1,0),\mathbb{Q}\big)
\ed
is not surjective.

To do this, we consider the $\mathbb{C}^*$-action on $(\mathbb{C}^*\times\mathbb{C})^{[n]}$ defined by scaling in the $\mathbb{C}$-factor, as in \cite{Grojnowski} or \cite{Nakajima}.  We use Chapter 7 of \cite{Nakajima} as our reference.  This action is {\em elliptic} in the sense used in \cite{BDMN}: that is, all downward flows converge, and thus one obtains a Bia\l ynicki-Birula (henceforth, BB) decomposition.  

More precisely, we abbreviate $N = (\mathbb{C}^*\times\mathbb{C})^{[n]}$.  Recall that 
$\on{Sym}^m(\mathbb{C}^*) \cong \mathbb{C}^*\times \mathbb{C}^{m-1}$: the first coordinate is the product of the $m$ elements of $\mathbb{C}^*$, and the remaining coordinates are (up to signs) the remaining elementary symmetric functions of the $m$ scalars.

For a partition $\lambda = 1^{\lambda_1} 2^{\lambda_2} 3^{\lambda_3}\dots$, we have the symmetric product
\begin{equation}\label{symmetric product}
S^\lambda(\mathbb{C}^*) = S^{\lambda_1}(\mathbb{C}^*)\times S^{\lambda_2}(\mathbb{C}^*)\times\dots
\cong \prod_{\lambda_i>0} (\mathbb{C}^*\times\mathbb{C}^{\lambda_i-1}).
\end{equation}
Then the fixed locus $N^{\mathbb{C}^*}$ is isomorphic to the disjoint union, 
\bd
N^{\mathbb{C}^*} \cong \bigsqcup_{\lambda: |\lambda| =n} S^\lambda(\mathbb{C}^*),
\ed 
and writing 
\bd
S_\lambda  = \big\{x\in N \; \big| \; \lim_{t\rightarrow 0} t\cdot x \in S^\lambda(\mathbb{C}^*)\big\},
\ed
we get $N = \bigsqcup_{\lambda} S_\lambda$.  Then 
\begin{equation}\label{BB decomp}
H^*(N, \mathbb{Q})\cong \bigoplus_{\lambda} H^*(S_\lambda)
\end{equation} (a BB decomposition; we ignore grading shifts).  If $\Gamma$ is any finite group acting by automorphisms of $N$ commuting with the $\mathbb{C}^*$-action, then $\Gamma$ acts naturally on the left-hand and right-hand sides of \eqref{BB decomp}, and the splittings involved in choosing a BB decomposition can be chosen $\Gamma$-equivariantly to make \eqref{BB decomp} an isomorphism of $\Gamma$-representations.

The $\mathbb{C}^*$-component of the ``center-of-mass'' map, i.e.,  $\det: N \rightarrow \mathbb{C}^*\times \mathbb{C} \xrightarrow{\pi_1} \mathbb{C}^*$, restricts to $S^\lambda(\mathbb{C}^*)$ as the projection on the product $\ds\prod_{i=1}^n \mathbb{C}^*$ of $\mathbb{C}^*$-factors of \eqref{symmetric product} followed by the map
\bd
\prod_{i=1}^n \mathbb{C}^* \rightarrow \mathbb{C}^*, \hspace{2em} (x_1,\dots, x_n) \mapsto x_1\cdot x_2^2 \cdot x_3^3 \cdot \dots\cdot x_n^n.
\ed

Now consider the action of the group $\Gamma = \mu_n$ of $n$th roots of unity, identified with the center of  $SL_n$, on $N$ by left multiplication on $SL_n$.  This action extends to an action of the connected group $\mathbb{C}^*$ on $GL_n$, hence on $N$, which thus acts trivially on $H^*(N)$.  Considering the $\Gamma$-action on 
$\mathsf{M}/\!\!/\!\!/ K$,
we find that, for $\lambda = (n) = (1^0 2^0 \dots n^1)$, we have that $\on{det}^{-1}(1)\cap S^{\lambda}(\mathbb{C}^*)$ is in natural bijection with $\Gamma$.  In other words: 
\begin{prop}
The set of length $n$ subschemes that are $\mathbb{C}^*$-fixed and have the form $\{\xi\}\times \on{Spec}\mathbb{C}[t]/(t^n)\subset \mathbb{C}^*\times\mathbb{C}$ for some $\xi\in\Gamma$, form a collection of connected components of
$\big((\on{det}\times\on{tr})^{-1}(1,0)\big)^{\mathbb{C}^*}$.
\end{prop}
Since the action of $\Gamma$ on this set of components obviously freely cyclically permutes them, we find that the regular representation of $\Gamma$ appears as a subrepresentation of $H^*(\mathsf{M}/\!\!/\!\!/ K, \mathbb{Q})$, thus completing the proof of Theorem \ref{main thm}.\hfill\qedsymbol
\begin{remark}
As in Remark 7.6 of \cite{CNS}, the proof above actually shows (as asserted in Theorem \ref{main thm}) that the regular representation $\mathbb{Q}[\Gamma]$ actually appears in the {\em pure cohomology} (in the Hodge-theoretic sense) $\ds\bigoplus_k W_kH^k(\mathsf{M}/\!\!/\!\!/ K, \mathbb{Q})$: the Bia\l ynicki-Birula decomposition is compatible with Hodge weights, and we have identified the regular representation in the pure part of the cohomology of the $\mathbb{C}^*$-fixed locus.  
\end{remark}

\section{Appendix: Some Hermitian Linear Algebra}
This section proves an elementary result about families (over the interval $[0,1]\subset\mathbb{R}$) of self-adjoint operators on a finite-dimensional Hermitian vector space needed in the proof of Theorem \ref{kempf-ness intro}.  While much stronger results are available in the literature, we include a proof of what we need, to emphasize to the more algebro-geometrically inclined reader that no sophisticated real analysis is needed.

Fix a complex vector space $V$ of dimension $n$ with Hermitian metric $\langle \cdot,\cdot\rangle$.  Let $L\in \mathfrak{gl}(V)[t]$ be a polynomial map $\mathbb{A}^1_{\mathbb{C}}\rightarrow\mathfrak{gl}(V)$ for which $L(\mathbb{R})\subseteq i\mathfrak{u}(V)$, the space of self-adjoint operators on $V$.  
\begin{prop}\label{continuous diagonalization}
There exist continuous maps $v_i(t): \mathbb{R}\rightarrow V$, $i=1,\dots, n$, and continuous functions 
 $\eta_i(t): \mathbb{R}\rightarrow \mathbb{R}$
such that:
\begin{enumerate}
\item The vectors $v_1(t),\dots, v_n(t)$ form a $\mathbb{C}$-linear basis of $V$ for each $t\in \mathbb{R}$, orthogonal with respect to the Hermitian inner product. 
\item  $L(t) v_i(t) = \eta_i(t) v_i(t)$ for all $i$ and $t$.
\end{enumerate}
\end{prop}
\begin{proof}
Write $C=\mathbb{A}^1_{\mathbb{C}}$ for the domain of the morphism $L$, with coordinate $t$.  
Taking the characteristic polynomial defines a polynomial map $\on{char}(L): C\rightarrow \on{Sym}^{n}(\mathbb{A}^1_{\mathbb{C}}) \cong \mathbb{A}^{n}_{\mathbb{C}}$.  
\begin{remark}\label{defined over R}
The ramified covering $c:\mathbb{A}^n_{\mathbb{C}} \rightarrow \on{Sym}^n(\mathbb{A}^1_{\mathbb{C}})$, as well as its restriction to every intersection of reflection hyperplanes, is defined over $\mathbb{R}$.
\end{remark}
There exists a finite subset $S\subset C$ such that $L|_{C\smallsetminus S}$ has a constant number of distinct (generalized) eigenvalues and the set of their multiplicities is constant. We now pass to a finite covering $\wt{C}  = \mathbb{A}^1_{\mathbb{C}}\xrightarrow{\pi} \mathbb{A}^1_{\mathbb{C}} = C$, $\pi(u) = t$, obtained by pulling back $c$ along $\on{char}(L)$.  The covering $\pi$ is ramified at most over $S$, and there are polynomials $\eta_i(u): \wt{C}\rightarrow\mathbb{C}$, $i=1,\dots, n$  giving the generalized eigenvalues of $L(u) := L(\pi(u))$.   Write $\wt{S} = \pi\inv(S)$.
Let $\ds D = \prod_{p\in \wt{S}}(u-p)$ and let $R = \mathbb{C}[u][D\inv]$.

Each linear operator $L_i = (L(u)-\eta_i(u)\on{Id})^n$ has constant rank $r_i$ as a function of $u\in \wt{C}\smallsetminus \wt{S}$, and thus $K_i := \on{ker}\big(L_i(u): V\otimes R\rightarrow V\otimes R\big)$ is a projective $R$-submodule of $V\otimes R$ of rank $n-r_i$.  Such a submodule is of the form $K_i = \overline{K}_i\otimes_{\mathbb{C}[u]} R$ for a submodule $\overline{K}_i\subset V[u]$ uniquely determined by the properties that $K_i = \overline{K}_i\otimes_{\mathbb{C}[u]} R$ and that $V[u]/\overline{K}_i$ is torsion-free.
By the classification of modules over a PID, we may choose an isomorphism $\overline{K}_i\cong \mathbb{C}[u]^{n-r_i}$, thus yielding a basis of $\overline{K}_i$; since every element $b(u)$ of this basis satisfies $(L(u) - \eta_i(u))^nb(u) = 0$ for $u\in \wt{C}\smallsetminus \wt{S}$, $b(u)$ is a generalized $\eta_i(u)$-eigenvector over all of $\wt{C}$.

Repeating the previous paragraph for all $\eta_i(u)$, we thus get a basis $w_1(u), \dots, w_n(u) \in V[u]$ so that $(L(u) - \eta_i(u))^nw_i(u) =0$ for $i=1,\dots, n$.  Apply Gram-Schmidt to the basis $\{w_i(u)\}$ to obtain an orthogonal basis that depends polynomially on $u$ and $\overline{u}$; we write $\{v_i(u,\overline{u})\}$ for this basis, and $\{v_i(u)\}$ for the basis restricted to $u\in\mathbb{R}$, which depends polynomially on $u\in\mathbb{R}$.    Since $L(u)$ is self-adjoint for $u\in \mathbb{R}$, we have that $L(u) v_i(u) = \eta_i(u)v_i(u)$ for all $u\in\mathbb{R}$.  

Finally, the ramified cover $\wt{C}\rightarrow C$, restricted to the real curve $\on{char}(L)(t)$, $t\in \mathbb{R}$, has a continuous section; pulling back the $v_i(u)$ and $\eta_i(u)$ gives the claimed assertion.
\end{proof}

\end{document}